\documentclass[a4paper,12pt]{amsart}

\usepackage{amsfonts}
\usepackage{amsmath}
\usepackage{amssymb}
\usepackage{graphicx}

\usepackage[usenames]{color}
\usepackage[colorlinks]{hyperref}

\newtheorem{thm}{Theorem}[section]
\newtheorem{lem}[thm]{Lemma}

\newtheorem{assum}[thm]{Assumption}

\theoremstyle{definition}

\theoremstyle{remark}
\newtheorem{rem}{Remark}[section]
\newtheorem{defn}{Definition}

\numberwithin{equation}{section}

\begin{document}

\title[Schr\"odinger equation with singular position dependent mass]{Schr\"odinger equation with singular position dependent mass}

\author[M. Ruzhansky]{Michael Ruzhansky}
\address{
  Michael Ruzhansky:
  \endgraf
  Department of Mathematics: Analysis, Logic and Discrete Mathematics
  \endgraf
  Ghent University, Krijgslaan 281, Building S8, B 9000 Ghent
  \endgraf
  Belgium
  \endgraf
  and
  \endgraf
  School of Mathematical Sciences
  \endgraf
  Queen Mary University of London
  \endgraf
  United Kingdom
  \endgraf
  {\it E-mail address} {\rm michael.ruzhansky@ugent.be}
}

\author[M. Sebih]{Mohammed Elamine Sebih}
\address{
  Mohammed Elamine Sebih:
  \endgraf
  Laboratory of Geomatics, Ecology and Environment (LGEO2E)
  \endgraf
  University Mustapha Stambouli of Mascara, 29000 Mascara
  \endgraf
  Algeria
  \endgraf
  {\it E-mail address} {\rm sebihmed@gmail.com, ma.sebih@univ-mascara.dz}
}

\author[N. Tokmagambetov ]{Niyaz Tokmagambetov }
\address{
  Niyaz Tokmagambetov:
  \endgraf 
  Centre de Recerca Matem\'atica
  \endgraf
  Edifici C, Campus Bellaterra, 08193 Bellaterra (Barcelona), Spain
  \endgraf
  and
  \endgraf   
  Institute of Mathematics and Mathematical Modeling
  \endgraf
  125 Pushkin str., 050010 Almaty, Kazakhstan
  \endgraf  
  {\it E-mail address:} {\rm tokmagambetov@crm.cat; tokmagambetov@math.kz}
  }

\thanks{This research was partly funded by the Committee of Science of the Ministry of Science and Higher Education of the Republic of Kazakhstan (Grant No. AP14872042). The authors are supported by the FWO Odysseus 1 grant G.0H94.18N: Analysis and Partial Differential Equations and by the Methusalem programme of the Ghent University Special Research Fund (BOF) (Grant number 01M01021). Michael Ruzhansky is also supported by EPSRC grant
EP/R003025/2.}


\begin{abstract}
We consider the Schr\"odinger equation with singular position dependent effective mass and prove that it is very weakly well posed. A uniqueness result is proved in an appropriate sense, moreover, we prove the consistency with the classical theory. In particular, this allows one to consider $\delta$-like or more singular masses.
\end{abstract}

\maketitle


\section{Introduction}
In the present paper we consider the Cauchy problem
\begin{equation}
    \left\lbrace
    \begin{array}{l}
    iu_{t}(t,x) + \sum_{j=1}^{d}\partial_{x_{j}}\left(g(x)\partial_{x_{j}}u(t,x)\right)=0, \,\,\,(t,x)\in\left[0,T\right]\times \mathbb{R}^{d},\\
    u(0,x)=u_{0}(x), \,\,\, x\in\mathbb{R}^{d}, \label{Equation intro}
    \end{array}
    \right.
\end{equation}
where the coefficient $g$ is a positive real valued function or distribution, and $d\geq 1$. We also assume that $g$ is irregular, that is discontinuous or even having strong singularities, distributional or non-distributional.

This kind of problems are investigated in the literature also as an abstract mathematical problem, see for instance \cite{Ban03,BIZ16,BP06,Ich87,Kaj98,Sal05}, where the coefficient $g$ refers to a non Euclidean metric and the Laplace-Beltrami operator
\begin{equation*}
    \Delta_{g} = \sum_{j=1}^{d}\partial_{x_{j}}\left(g(x)\partial_{x_{j}}\right),
\end{equation*}
is considered instead of the Laplace operator. In these works a certain regularity conditions are required on the coefficient. The problem is also investigated as a physical model where
\begin{equation*}
    g(x)=\frac{1}{m(x)},
\end{equation*}
and $m$ represents the mass when studying quantum systems with position dependent effective mass (PDEM). The Schr\"odinger equation with PDEM was initially proposed in \cite{Roos83,RM85} as a model governing the motion of electrons in graded mixed semiconductors where the need of a varying mass is natural. It was later used to describe many physical systems, such as materials of nonuniform chemical composition, supper-lattice band structures and abrupt heterojunctions. For more general overview about the PDEM formalism and its application we refer to \cite{BSP20,Don11,JER11,MJS03,Nab20,PMGA17,SE16,Youn89} and the references therein. In the present work, we are interested in both (mathematical and physical model) and we allow the coefficient $g$ to be irregular, even to have strong singularities. We note here, that in the case of, e.g. mixed graded semiconductors or abrupt heterojunctions where the mass depends on the position and can change suddenly, it is natural to assume the mass to be singular, and this motivates our intention to consider singular masses.

When studying problems such as (\ref{Equation intro}), many questions arise naturally:

\begin{itemize}
    \item To study the well-posedness which needs to give a meaningful notion of solution, since in the case of distributional data where the solutions are expected to have the same regularity as the data and thus nonlinear operations as the multiplication are not allowed in view of the celebrated result of Schwartz \cite{Sch54} about the impossibility of multiplication of distributions, and the classical theory fails. There are several ways to overcome this problem. One way is to use the framework of very weak solutions introduced recently in \cite{GR15} for the analysis of second order hyperbolic equations with irregular coefficients and used later in a series of papers to show a wide applicability, where we can refer to \cite{RT17a,RT17b,MRT19,ART19,Gar20,ARST21a,ARST21b,CRT21,CRT22a,CRT22b}.
    In contrast with the framework of the Colombeau algebras, see for instance \cite{Col84,Col85,Ober92}, where the consistency with classical solutions maybe lost in the case of non-smooth functions, the concept of very weak solutions which depends heavily on the equation under consideration is consistent with the classical theory. Also, the advantage in using the concept of very weak solutions, is that we are allowed to consider coefficients with strong singularities, while in the weak or distributional setting, we are limited to a certain stronger regularity.
    \item Another question that arises naturally when studying problems of the kind of (\ref{Equation intro}) is to study either analytically or numerically the influence of the singularities of the coefficients on the ``solution''. In \cite{GHO13,GO16}, the authors considered a wave equation with coefficients having jumps in the propagation speed and studied the propagation of singularities in the Colombeau setting. In the case of the very weak solutions context, we refer to \cite{SW20}, where the authors developed diagonalisation techniques to understand the influence of the singularities of the coefficients on the very weak solutions. From an experimental point of view, we refer to \cite{ART19,ARST21a,ARST21b,MRT19} where the problem of propagation of singularities was treated numerically.
\end{itemize}

In the present paper we consider the Cauchy problem (\ref{Equation intro}) and study its very weak well-posedness. Moreover, we prove the uniqueness of the very weak solution in an appropriate sense and the consistency with classical solutions.

Our purpose here is to contribute to the study of the well posedness, when the classical theory fails, of the Cauchy problem (\ref{Equation intro}) in the case of irregular coefficients by using the framework of very weak solutions and at the same time to show the wide applicability of the concept. Another aim is to show that the concept of very weak solutions is useful and easy to use in applications. 

The paper is organized as follows. In Section 2, after some preliminaries, we prove a fundamental lemma which is key to proving existence and uniqueness of a very weak solution and consistency with classical solutions. In Section 3, we introduce a notion of very weak solution to our considered model (\ref{Equation intro}) and prove that it is very weak well-posed. In Section 4, we prove that the very weak solution is consistent with the classical solutions when the latter exists.

\section{Preliminaries}
To start with, let us define some notions and notations that we use throughout this paper. Firstly, the notation $f\lesssim g$ means that there exists a positive constant $C$ such that $f \leq Cg$. In addition, we introduce the Sobolev space $W^{1,\infty}(\mathbb{R}^d)$ defined by
\begin{equation*}
    W^{1,\infty}(\mathbb{R}^d):=\left\{ f \text{\,is measurable:}\, \Vert f\Vert_{W^{1,\infty}}:= \Vert f\Vert_{L^{\infty}} + \Vert \nabla f\Vert_{L^{\infty}} < +\infty \right\}.
\end{equation*}
We will also denote by $\Vert \cdot\Vert_{H^2}$ the following norm
\begin{equation*}
        \Vert u(t,\cdot)\Vert_{H^2} := \Vert u(t,\cdot)\Vert_{L^2} + \Vert\sum_{j=1}^{d} \partial_{x_j} u(t,\cdot)\Vert_{L^2} + \Vert \Delta u(t,\cdot)\Vert_{L^2}.
    \end{equation*}

As mentioned above, we want to study the Cauchy problem
\begin{equation}
    \left\lbrace
    \begin{array}{l}
    iu_{t}(t,x) + \sum_{j=1}^{d}\partial_{x_{j}}\left(g(x)\partial_{x_{j}}u(t,x)\right)=0, \,\,\,(t,x)\in\left[0,T\right]\times \mathbb{R}^{d},\\
    u(0,x)=u_{0}(x), \,\,\, x\in\mathbb{R}^{d}. \label{Equation}
    \end{array}
    \right.
\end{equation}
In the case when the coefficient and the Cauchy data in (\ref{Equation}) are regular functions, we easily prove the following lemma which is the key to proving the well-posedness of our problem.

\begin{lem} \label{lem1}
Let $u_{0}\in L^2(\mathbb{R}^{d})$ and assume that $g\in L^{\infty}(\mathbb{R}^d)$ satisfies $\inf_{x\in \mathbb{R}^d}g(x)>0$. Then, the energy conservation
\begin{equation}
    \Vert u(t,\cdot)\Vert_{L^2} = \Vert u_{0}\Vert_{L^2}~, \label{Energy estimate 1}
\end{equation}
holds for all $t\in [0,T]$, for the unique solution $u\in C(\left[0,T\right];L^2(\mathbb{R}^d))$ to the Cauchy problem (\ref{Equation}).
Moreover, if $g\in W^{1,\infty}(\mathbb{R}^d)$ and $u_0 \in H^{2}(\mathbb{R}^d)$, then, the solution $u\in C(\left[0,T\right];H^2(\mathbb{R}^d))$ satisfies the estimate
\begin{equation}
    \Vert u(t,\cdot)\Vert_{H^2} \lesssim \left(1 + \Vert g\Vert_{W^{1,\infty}}\right) \Vert u_0\Vert_{H^2}~, \label{Energy estimate 2}
\end{equation}
for all $t\in [0,T]$.
\end{lem}

\begin{proof}
We multiply the equation in (\ref{Equation}) by $-i$, to get
\begin{equation}
    u_{t}(t,x) - i\sum_{j=1}^{d}\partial_{x_{j}}\left(g(x)\partial_{x_{j}}u(t,x)\right)=0. \label{Equation lem}
\end{equation}
Multiplying the obtained equation by $u$, integrating over $\mathbb{R}^d$ and taking the real part, we obtain
\begin{equation*}
    Re \left(\langle u_{t}(t,\cdot),u(t,\cdot)\rangle_{L^2} -i \langle \sum_{j=1}^{d}\partial_{x_{j}}\left(g(\cdot)\partial_{x_{j}}u(t,\cdot)\right),u(t,\cdot)\rangle_{L^2} \right)=0,
\end{equation*}
where we easily see that
\begin{equation*}
    Re \langle u_{t}(t,\cdot),u(t,\cdot)\rangle_{L^2} = \frac{1}{2}\partial_{t} \Vert u(t,\cdot)\Vert_{L^2}^{2},
\end{equation*}
and that
\begin{equation*}
    Re \left(i \langle \sum_{j=1}^{d}\partial_{x_{j}}\left(g(\cdot)\partial_{x_{j}}u(t,\cdot)\right),u(t,\cdot)\rangle_{L^2} \right)=0.
\end{equation*}
Thus, $\partial_{t} \Vert u(t,\cdot)\Vert_{L^2}^{2}=0$ and our first statement follows. Let us now assume that $g\in W^{1,\infty}(\mathbb{R}^d)$ and $u_0 \in H^{2}(\mathbb{R}^d)$ and prove our second statement. By the fundamental theorem of calculus we have that
\begin{equation}
    u(t,x) = u_0(x) + \int_{0}^{t} u_{t}(s,x) ds. \label{Fund. th. calculus}
\end{equation}
Taking the $L^2$ norm in (\ref{Fund. th. calculus}) we get the estimate
\begin{equation}
    \Vert u(t,\cdot)\Vert_{L^2} \lesssim \Vert u_0\Vert_{L^2} + \int_{0}^{T} \Vert u_t(s,\cdot) \Vert_{L^2} ds. \label{Estimate u}
\end{equation}
We know that if $u$ solves the Cauchy problem
\begin{equation}
    \left\lbrace
    \begin{array}{l}
    iu_{t}(t,x) + \sum_{j=1}^{d}\partial_{x_{j}}\left(g(x)\partial_{x_{j}}u(t,x)\right)=0, \,\,\,(t,x)\in\left[0,T\right]\times \mathbb{R}^{d},\\
    u(0,x)=u_{0}(x), \,\,\, x\in\mathbb{R}^{d},
    \end{array}
    \right.
\end{equation}
then $u_t$ solves
\begin{equation}
    \left\lbrace
    \begin{array}{l}
   i\partial_t u_{t}(t,x) + \sum_{j=1}^{d}\partial_{x_{j}}\left(g(x)\partial_{x_{j}}u_t(t,x)\right)=0, \,\,\,(t,x)\in\left[0,T\right]\times \mathbb{R}^{d},\\
    u_t(0,x)=i\sum_{j=1}^{d}\partial_{x_{j}}\left(g(x)\partial_{x_{j}} u_0(x)\right), \,\,\, x\in\mathbb{R}^{d}. \label{Equation u_t}
    \end{array}
    \right.
\end{equation}
Using the conservation law (\ref{Energy estimate 1}) for the solution $u_t$ to the Cauchy problem (\ref{Equation u_t}), we obtain
\begin{align}
    \Vert u_t(t,\cdot)\Vert_{L^2} & \lesssim \sum_{j=1}^{d}\Vert \partial_{x_{j}}\left(g(\cdot)\partial_{x_{j}} u_0(\cdot)\right)\Vert_{L^2} \label{Estimate u_t} \\
    & \lesssim \sum_{j=1}^{d}\Vert \partial_{x_{j}} g(\cdot)\partial_{x_{j}} u_0(\cdot)\Vert_{L^2} + \sum_{j=1}^{d}\Vert g(\cdot)\partial_{x_{j}}^{2} u_0(\cdot)\Vert_{L^2}\nonumber \\ 
    & \lesssim \Vert \nabla g\Vert_{L^{\infty}}\Vert u_0\Vert_{H^1} + \Vert g\Vert_{L^{\infty}}\Vert u_0\Vert_{H^2}\nonumber \\
    & \lesssim \Vert g\Vert_{W^{1,\infty}}\Vert u_0\Vert_{H^2}. \nonumber
\end{align}
In the last inequality, we used the fact that for all $j=1,\dots,d$, the terms $\Vert \partial_{x_{j}} g(\cdot)\partial_{x_{j}} u_0(\cdot)\Vert_{L^2}$ and $\Vert g(\cdot)\partial_{x_{j}}^{2} u_0(\cdot)\Vert_{L^2}$ can be estimated by $\Vert \partial_{x_{j}} g(\cdot)\Vert_{L^{\infty}}\Vert\partial_{x_{j}} u_0(\cdot)\Vert_{L^2}$ and $\Vert g(\cdot)\Vert_{L^{\infty}}\Vert\partial_{x_{j}}^{2} u_0(\cdot)\Vert_{L^2}$ respectively.
From (\ref{Estimate u}) we obtain the following estimate for $u$
\begin{equation}
    \Vert u(t,\cdot)\Vert_{L^2} \lesssim \left(1 + \Vert g\Vert_{W^{1,\infty}}\right) \Vert u_0\Vert_{H^2}~. \label{Estimate1 u}
\end{equation}
Now, if we multiply the equation in (\ref{Equation}) by $u_t$ and once again we integrate over $\mathbb{R}^d$ and we take the real part, we get
\begin{equation}
    Re \left(\langle iu_{t}(t,\cdot),u_{t}(t,\cdot)\rangle_{L^2} + \langle \sum_{j=1}^{d}\partial_{x_{j}}\left(g(\cdot)\partial_{x_{j}}u(t,\cdot)\right),u_{t}(t,\cdot)\rangle_{L^2} \right)=0. \label{Formula 1}
\end{equation}
Noting that
\begin{equation*}
    Re \langle iu_{t}(t,\cdot),u_{t}(t,\cdot)\rangle_{L^2}=0,
\end{equation*}
and that
\begin{equation*}
    Re \langle \sum_{j=1}^{d}\partial_{x_{j}}\left(g(\cdot)\partial_{x_{j}}u(t,\cdot)\right),u_{t}(t,\cdot)\rangle_{L^2} = -\frac{1}{2}\partial_{t}\sum_{j=1}^{d}\Vert g^{\frac{1}{2}}\partial_{x_{j}}u(t,\cdot)\Vert_{L^2}^{2},
\end{equation*}
it follows from (\ref{Formula 1}), that the quantity $\sum_{j=1}^{d}\Vert g^{\frac{1}{2}}\partial_{x_{j}}u(t,\cdot)\Vert_{L^2}^2$ is conserved in time. That is
\begin{equation*}
    \sum_{j=1}^{d}\Vert g^{\frac{1}{2}}\partial_{x_{j}}u(t,\cdot)\Vert_{L^2}^2 = \sum_{j=1}^{d}\Vert g^{\frac{1}{2}}\partial_{x_{j}}u_{0}\Vert_{L^2}^2.
\end{equation*}
By the remark that for all $j=1,\dots,d$, the terms $\Vert g^{\frac{1}{2}}\partial_{x_i}u_0\Vert_{L^2}^2$ in the right hand side of the last equality can be estimated by
\begin{equation*}
\Vert g^{\frac{1}{2}}\partial_{x_j}u_0\Vert_{L^2}^2 \leq \Vert g\Vert_{L^{\infty}}\Vert u_0\Vert_{H^1}^2,
\end{equation*}
we get for all $j=1,\dots,d$,
\begin{equation*}
    \Vert g^{\frac{1}{2}}\partial_{x_{j}}u(t,\cdot)\Vert_{L^2} \lesssim \Vert g\Vert_{L^{\infty}}^{\frac{1}{2}}\Vert u_0\Vert_{H^1}.
\end{equation*}
Now, using the assumption that $g$ is bounded from below, that is,
\begin{equation*}
    \inf\limits_{x\in\mathbb R^d}g(x) = c_0 > 0,
\end{equation*}
we get
\begin{equation}
    \Vert \partial_{x_{j}}u(t,\cdot)\Vert_{L^2} \lesssim \Vert g\Vert_{L^{\infty}}^{\frac{1}{2}}\Vert u_0\Vert_{H^1} \leq \left(1 + \Vert g\Vert_{W^{1,\infty}}\right) \Vert u_0\Vert_{H^2}, \label{Estimate du_xj}
\end{equation}
for all $j=1,\dots,d$.
Let us now prove the estimate for $\Delta u$. Taking the $L^2$ norm in the equality
\begin{equation*}
    \sum_{j=1}^{d}g(x)\partial_{x_{j}}^{2}u(t,x) = -iu_{t}(t,x) - \sum_{j=1}^{d}\partial_{x_{j}}g(x)\partial_{x_{j}}u(t,x),
\end{equation*}
obtained from the equation in (\ref{Equation}) and using the estimates (\ref{Estimate u_t}) and (\ref{Estimate du_xj}) and the estimate
\begin{equation*}
    c_0^{\frac{1}{2}}\Vert \Delta u(t,\cdot)\Vert_{L^2} \leq \Vert g(\cdot)\Delta u(t,\cdot)\Vert_{L^2} = \Vert \sum_{j=1}^{d}g(\cdot)\partial_{x_{j}}^{2}u(t,\cdot)\Vert_{L^2},
\end{equation*}
resulting from the assumption that $g$ is positive, we arrive at
\begin{equation}
    \Vert \Delta u(t,\cdot)\Vert_{L^2} \lesssim \left(1 + \Vert g\Vert_{W^{1,\infty}}\right) \Vert u_0\Vert_{H^2}. \label{Estimate Du}
\end{equation}
The second statement of our lemma follows by summing the so far proved estimates (\ref{Estimate1 u}), (\ref{Estimate du_xj}) and (\ref{Estimate Du}). This concludes the proof.
\end{proof}

\section{Very weak well-posedness}
In what follows we consider that the coefficient $g$ and the Cauchy data $u_0$ are singular and we want to prove that a unique very weak solution exists for the Cauchy problem 
\begin{equation}
    \left\lbrace
    \begin{array}{l}
    iu_{t}(t,x) + \sum_{j=1}^{d}\partial_{x_{j}}\left(g(x)\partial_{x_{j}}u(t,x)\right)=0, \,\,\,(t,x)\in\left[0,T\right]\times \mathbb{R}^{d},\\
    u(0,x)=u_{0}(x), \,\,\, x\in\mathbb{R}^{d}, \label{Equation sing}
    \end{array}
    \right.
\end{equation}
Following the idea from \cite{GR15}, we start by regularising $g$ and $u_0$ by convolution with a suitable mollifier $\psi$, which generate families of smooth functions $(g_{\varepsilon})_{\varepsilon}$ and $(u_{0,\varepsilon})_{\varepsilon}$, that is
\begin{equation}
    g_{\varepsilon}(x) = g\ast \psi_{\varepsilon}(x)
\end{equation}
and
\begin{equation}
    u_{0,\varepsilon}(x) = u_0\ast \psi_{\varepsilon}(x),
\end{equation}
where
\begin{equation}
    \psi_{\varepsilon}(x) = \varepsilon^{-1}\psi(x/\varepsilon),\,\,\,\varepsilon\in\left(0,1\right].
\end{equation}
The function $\psi$ is a Friedrichs-mollifier, i.e. $\psi\in C_{0}^{\infty}(\mathbb{R}^{d})$, $\psi\geq 0$ and $\int\psi =1$. Let us now consider nets of smooth functions and introduce the notion of moderateness.

\begin{defn}[Moderateness] \label{defn:Moderatness}
\leavevmode
\begin{itemize}
    \item[(i)]  A net of functions $(f_{\varepsilon})_{\varepsilon}$, is said to be $W^{1,\infty}$-moderate, if there exist $N\in\mathbb{N}_{0}$ such that
\begin{equation*}
    \Vert f_{\varepsilon}\Vert_{W^{1,\infty}} \lesssim \varepsilon^{-N}.
\end{equation*}
    \item[(ii)] A net of functions $(g_{\varepsilon})_{\varepsilon}$, is said to be $H^2$-moderate, if there exist $N\in\mathbb{N}_{0}$ such that
\begin{equation*}
    \Vert g_{\varepsilon}\Vert_{H^2} \lesssim \varepsilon^{-N}.
\end{equation*}
    \item[(iii)] A net of functions $(u_{\varepsilon})_{\varepsilon}$ from $C([0,T]; H^{2}(\mathbb R^d))$ is said to be $C$-moderate, if there exist $N\in\mathbb{N}_{0}$ such that
\begin{equation*}
    \sup_{t\in[0,T]}\Vert u_{\varepsilon}(t,\cdot)\Vert_{H^2} \lesssim \varepsilon^{-N}.
\end{equation*}
\end{itemize}
\end{defn}

\begin{rem}
We note that by regularising a distribution $V\in \mathcal{E}'(\mathbb{R}^{d})$ by convolution with a mollifier $\psi_{\varepsilon}$ as defined above we get a net of moderate functions. Indeed, by the structure theorems for distributions (see, e.g. \cite{FJ98}), we know that every compactly supported distribution can be represented by a
finite sum of (distributional) derivatives of continuous functions. Precisely, for $V\in \mathcal{E}'(\mathbb{R}^{d})$ we can find $n\in \mathbb{N}$ and functions $f_{\alpha}\in C(\mathbb{R}^{d})$ such that, $V=\sum_{\vert \alpha\vert \leq n}\partial^{\alpha}f_{\alpha}$. The convolution of $V$ with a mollifier gives
\begin{equation}
    V\ast\psi_{\varepsilon}=\sum_{\vert \alpha\vert \leq n}\partial^{\alpha}f_{\alpha}\ast\psi_{\varepsilon}=\sum_{\vert \alpha\vert \leq n}f_{\alpha}\ast\partial^{\alpha}\psi_{\varepsilon}=\sum_{\vert \alpha\vert \leq n}\varepsilon^{-\vert\alpha\vert}f_{\alpha}\ast\left(\varepsilon^{-1}\partial^{\alpha}\psi(x/\varepsilon)\right).
\end{equation}
Using an appropriate norm, we see that the net $(V_{\varepsilon})_{\varepsilon}=(V\ast\psi_{\varepsilon})_{\varepsilon}$ is moderate.
\end{rem}

Let us now give our main assumptions. We note that for the sake of generality we will take assumptions on the regularisations of the coefficient and the Cauchy data instead of taking them on the functions themselves.

\begin{assum}
In addition to the assumption that the coefficient $g$ is positive (in the sense that its regularisations $g_{\varepsilon}$ are positive and $\inf_{\varepsilon\in (0,1]}\inf_{x\in \mathbb{R}^d}g_{\varepsilon}(x)>0$), we assume that
\begin{itemize}
    \item[\textit{\textbf{(A1)}}] $g$ is $W^{1,\infty}$-moderate.  
    \item [\textit{\textbf{(A2)}}]$u_0$ is $H^2$-moderate.  
\end{itemize}
\end{assum}

\begin{rem}
We mention here that our assumptions are in agreement with the case when the coefficient $g$ is a distribution, for instance, when it has delta functions. This covers the physical case when, for example, the mass is positive and in some points is very close to zero. In that case the coefficient $g$ goes to infinity and could be approximated by delta functions or even stronger singularities.
\end{rem}

We are now ready to introduce a notion of very weak solution to the Cauchy problem (\ref{Equation sing}).

\begin{defn}[Very weak solution] \label{defn:very weak sol}
The net $(u_{\varepsilon})_{\varepsilon}\in C([0,T]; H^{2}(\mathbb R^d))$ is said to be a very weak solution to the Cauchy problem (\ref{Equation sing}), if there exist a $W^{1,\infty}$-moderate regularisation of the coefficient $g$ and a $H^2$-moderate regularisation of $u_0$, such that the net $(u_{\varepsilon})_{\varepsilon}$ solves the regularized problem
\begin{equation}
    \left\lbrace
    \begin{array}{l}
    i\partial_{t}u_{\varepsilon}(t,x) + \sum_{j=1}^{d}\partial_{x_{j}}\left(g_{\varepsilon}(x)\partial_{x_{j}}u_{\varepsilon}(t,x)\right)=0, \,\,\,(t,x)\in\left[0,T\right]\times \mathbb{R}^{d},\\
    u_{\varepsilon}(0,x)=u_{0,\varepsilon}(x), \,\,\, x\in\mathbb{R}^{d}, \label{Equation regularized}
    \end{array}
    \right.
\end{equation}
for all $\varepsilon\in\left(0,1\right]$, and is $C$-moderate.
\end{defn}

The uniqueness of a very weak solution to the Cauchy problem (\ref{Equation sing}) will be understood in the sense of the following definition.

\begin{defn}[Uniqueness] \label{defn uniq}
We say that the Cauchy problem (\ref{Equation sing}) has a unique very weak solution, if for all families of regularisations $(g_{\varepsilon})_{\varepsilon}$, $(\Tilde{g}_{\varepsilon})_{\varepsilon}$ and $(u_{0,\varepsilon})_{\varepsilon}$, $(\Tilde{u}_{0,\varepsilon})_{\varepsilon}$ of the coefficient $g$ and the Cauchy data $u_0$, satisfying
\begin{equation}
    \Vert g_{\varepsilon}-\Tilde{g}_{\varepsilon}\Vert_{W^{1,\infty}}\leq C_{k}\varepsilon^{k} \text{\,\,for all\,\,} k>0,
\end{equation}
and
\begin{equation}
    \Vert u_{0,\varepsilon}-\Tilde{u}_{0,\varepsilon}\Vert_{L^{2}}\leq C_{m}\varepsilon^{m} \text{\,\,for all\,\,} m>0,
\end{equation}
we have
\begin{equation*}
    \Vert u_{\varepsilon}(t,\cdot)-\Tilde{u}_{\varepsilon}(t,\cdot)\Vert_{L^{2}} \leq C_{N}\varepsilon^{N}
\end{equation*}
for all $N>0$,  
where $(u_{\varepsilon})_{\varepsilon}$ and $(\Tilde{u}_{\varepsilon})_{\varepsilon}$ are the nets of solutions to the related regularized Cauchy problems.
\end{defn}

\begin{thm}[Existence and uniqueness]
Let the coefficient $g$ be positive and assume \textbf{(A1)}, \textbf{(A2)}. Then, the Cauchy problem (\ref{Equation sing}) has a unique very weak solution.
\end{thm}

\begin{proof}
For given moderate coefficient $g$ and data $u_0$, to prove that a very weak solution to the Cauchy problem (\ref{Equation sing}) exists, we need to prove that the net $(u_{\varepsilon})_{\varepsilon}$, solution to the family of regularized Cauchy problems
\begin{equation}
    \left\lbrace
    \begin{array}{l}
    i\partial_{t}u_{\varepsilon}(t,x) + \sum_{j=1}^{d}\partial_{x_{j}}\left(g_{\varepsilon}(x)\partial_{x_{j}}u_{\varepsilon}(t,x)\right)=0, \,\,\,(t,x)\in\left[0,T\right]\times \mathbb{R}^{d},\\
    u_{\varepsilon}(0,x)=u_{0,\varepsilon}(x), \,\,\, x\in\mathbb{R}^{d},
    \end{array}
    \right.
\end{equation}
is $C$-moderate. Using the assumptions \textit{\textbf{(A1)}}, \textit{\textbf{(A2)}} we know that there exist $N_1,N_2\in\mathbb{N}_0$ such that
\begin{equation*}
    \Vert g_{\varepsilon}\Vert_{W^{1,\infty}} \lesssim \varepsilon^{-N_1},
\end{equation*}
and 
\begin{equation*}
    \Vert u_{0,\varepsilon}\Vert_{H^2} \lesssim \varepsilon^{-N_2}.
\end{equation*}
It follows from the energy estimate (\ref{Energy estimate 2}) that for all $t\in[0,T]$
\begin{equation*}
    \Vert u_{\varepsilon}(t,\cdot)\Vert_{H^2} \lesssim \varepsilon^{-2N_1-N_2},
\end{equation*}
which proves the existence. Let us now prove, in the sense of Definition \ref{defn uniq}, that the very weak solution is unique. Assume $(u_{\varepsilon})_{\varepsilon}$ and $(\Tilde{u}_{\varepsilon})_{\varepsilon}$ to be two solutions related to regularisations $(g_{\varepsilon}, u_{0,\varepsilon})_{\varepsilon}$ and $(\Tilde{g}_{\varepsilon}, \Tilde{u}_{0,\varepsilon})_{\varepsilon}$ satisfying
\begin{equation*}
    \Vert g_{\varepsilon}-\Tilde{g}_{\varepsilon}\Vert_{W^{1,\infty}}\leq C_{k}\varepsilon^{k} \text{\,\,for all\,\,} k>0,
\end{equation*}
and
\begin{equation*}
    \Vert u_{0,\varepsilon}-\Tilde{u}_{0,\varepsilon}\Vert_{L^{2}}\leq C_{m}\varepsilon^{m} \text{\,\,for all\,\,} m>0.
\end{equation*}
Let us denote by $U_{\varepsilon}(t,x):=u_{\varepsilon}(t,x)-\Tilde{u}_{\varepsilon}(t,x)$ and consider the auxiliary Cauchy problems
\begin{equation}
    \left\lbrace
    \begin{array}{l}
    i\partial_{t}V_{\varepsilon}(t,x) + \sum_{j=1}^{d}\partial_{x_{j}}\left(\Tilde{g}_{\varepsilon}(x)\partial_{x_{j}}V_{\varepsilon}(t,x)\right) = 0, \,\,\,(t,x)\in\left[0,T\right]\times \mathbb{R}^{d},\\
    V_{\varepsilon}(0,x)=(u_{0,\varepsilon}-\Tilde{u}_{0,\varepsilon})(x), \,\,\, x\in\mathbb{R}^{d},
    \end{array}
    \right.
\end{equation}
\begin{equation}
    \left\lbrace
    \begin{array}{l}
    i\partial_{t}W_{\varepsilon}(t,x;s) + \sum_{j=1}^{d}\partial_{x_{j}}\left(\Tilde{g}_{\varepsilon}(x)\partial_{x_{j}}W_{\varepsilon}(t,x;s)\right) = 0, \,\,\,(t,x)\in\left[0,T\right]\times \mathbb{R}^{d},\\
    W_{\varepsilon}(0,x;s)=f_{\varepsilon}(s,x), \,\,\, x\in\mathbb{R}^{d},
    \end{array}
    \right.
\end{equation}
where
\begin{equation}
    f_{\varepsilon}(t,x) = \sum_{j=1}^{d}\partial_{x_{j}}\left[\left(g_{\varepsilon}(x)-\Tilde{g}_{\varepsilon}(x)\right)\partial_{x_{j}}u_{\varepsilon}(t,x)\right]. \label{Formula for f_eps}
\end{equation}
Using Duhamel's principle (see, e.g. \cite{ER18}), $U_{\varepsilon}(t,x)$ can be represented by
\begin{equation}
    U_{\varepsilon}(t,x)=V_{\varepsilon}(t,x) + \int_{0}^{t}W_{\varepsilon}(x,t-s;s)ds. \label{Representation U_eps} 
\end{equation}
Now, we take the $L^2$ norm in both sides in (\ref{Representation U_eps}) and we use the energy conservation law (\ref{Energy estimate 1}) to estimate $V_{\varepsilon}$ and $W_{\varepsilon}$, to get
\begin{align}
    \Vert U_{\varepsilon}(\cdot,t)\Vert_{L^2} & \leq \Vert V_{\varepsilon}(\cdot,t)\Vert_{L^2} + \int_{0}^{T}\Vert W_{\varepsilon}(\cdot,t-s;s)\Vert_{L^2} ds \nonumber\\
    & \lesssim \Vert u_{0,\varepsilon}-\Tilde{u}_{0,\varepsilon}\Vert_{L^2} + \int_{0}^{T}\Vert f_{\varepsilon}(s,\cdot)\Vert_{L^2} ds. \label{Estimate U_eps}
\end{align}
To estimate $\Vert f_{\varepsilon}(s,\cdot)\Vert_{L^2}$ we use the product rule for derivatives in (\ref{Formula for f_eps}) and the fact that the terms $\Vert\partial_{x_{j}}\left(g_{\varepsilon}-\Tilde{g}_{\varepsilon}\right)\partial_{x_j}u_{\varepsilon}\Vert_{L^2}$ and $\Vert\left(g_{\varepsilon}-\Tilde{g}_{\varepsilon}\right)\partial_{x_j}^{2}u_{\varepsilon}\Vert_{L^2}$ can be estimated by $\Vert\partial_{x_{j}} g_{\varepsilon}-\partial_{x_{j}}\Tilde{g}_{\varepsilon}\Vert_{L^{\infty}} \Vert\partial_{x_j}u_{\varepsilon}\Vert_{L^2}$ and $\Vert g_{\varepsilon}-\Tilde{g}_{\varepsilon}\Vert_{L^{\infty}} \Vert\partial_{x_j}^{2}u_{\varepsilon}\Vert_{L^2}$, respectively, for all $j=1,...d$.
So we arrive at
\begin{align}
    \Vert f_{\varepsilon}(s,\cdot)\Vert_{L^2} & = \Vert\sum_{j=1}^{d}\partial_{x_{j}}\left[\left(g_{\varepsilon}(\cdot)-\Tilde{g}_{\varepsilon}(\cdot)\right)\partial_{x_{j}}u_{\varepsilon}(s,\cdot)\right]\Vert_{L^2}\nonumber\\
    & \leq \sum_{j=1}^{d} \Vert\partial_{x_{j}} g_{\varepsilon}-\partial_{x_{j}}\Tilde{g}_{\varepsilon}\Vert_{L^{\infty}} \Vert\partial_{x_j}u_{\varepsilon}\Vert_{L^2} + \Vert g_{\varepsilon}-\Tilde{g}_{\varepsilon}\Vert_{L^{\infty}}\Vert\sum_{j=1}^{d} \partial_{x_j}^{2}u_{\varepsilon}\Vert_{L^2}\nonumber\\
    & \leq \Vert g_{\varepsilon}-\Tilde{g}_{\varepsilon}\Vert_{W^{1,\infty}}\sum_{j=1}^{d} \Vert\partial_{x_j}u_{\varepsilon}\Vert_{L^2} + \Vert g_{\varepsilon}-\Tilde{g}_{\varepsilon}\Vert_{W^{1,\infty}}\Vert\sum_{j=1}^{d} \partial_{x_j}^{2}u_{\varepsilon}\Vert_{L^2}. \label{Estimate f_eps}
\end{align}
On the one hand, the net $(u_{\varepsilon})_{\varepsilon}$ is $C$-moderate as a very weak solution of the Cauchy problem (\ref{Equation sing}), that is, there exists $N\in \mathbb{N}_0$ such that
\begin{equation*}
    \sum_{j=1}^{d} \Vert\partial_{x_j}u_{\varepsilon}\Vert_{L^2} \lesssim \varepsilon^{-N},
\end{equation*}
and
\begin{equation*}
    \Vert\sum_{j=1}^{d} \partial_{x_j}^{2}u_{\varepsilon}\Vert_{L^2} \lesssim \varepsilon^{-N}.
\end{equation*}
From the other hand, we have by assumption that
\begin{equation*}
    \Vert g_{\varepsilon}-\Tilde{g}_{\varepsilon}\Vert_{W^{1,\infty}}\leq C_{k}\varepsilon^{k} \text{\,\,for all\,\,} k>0,
\end{equation*}
and
\begin{equation*}
    \Vert u_{0,\varepsilon}-\Tilde{u}_{0,\varepsilon}\Vert_{L^{2}}\leq C_{m}\varepsilon^{m} \text{\,\,for all\,\,} m>0.
\end{equation*}
It follows from (\ref{Estimate U_eps}) and (\ref{Estimate f_eps}) that
\begin{equation*}
    \Vert u_{\varepsilon}(t,\cdot)-\Tilde{u}_{\varepsilon}(t,\cdot)\Vert_{L^2} \lesssim \varepsilon^{n},
\end{equation*}
for all $n\in \mathbb{N}$. This proves the uniqueness of the very weak solution.
\end{proof}

\section{Consistency with classical solutions}
Our task here is to prove that when the Cauchy problem (\ref{Equation}) is well-posed in the classical sense, that means, when a classical solution $u\in C(\left[0,T\right];H^2(\mathbb{R}^d))$ exists, then every very weak solution $(u_{\varepsilon})_{\varepsilon}$ converges to $u$ in an appropriate norm.

\begin{thm}[Consistency]\label{Thm consistency}
Let us consider the Cauchy problem
\begin{equation}
    \left\lbrace
    \begin{array}{l}
    iu_{t}(t,x) + \sum_{j=1}^{d}\partial_{x_{j}}\left(g(x)\partial_{x_{j}}u(t,x)\right)=0, \,\,\,(t,x)\in\left[0,T\right]\times \mathbb{R}^{d},\\
    u(0,x)=u_{0}(x), \,\,\, x\in\mathbb{R}^{d}, \label{Equation consistency}
    \end{array}
    \right.
\end{equation}
where $g\in W^{1,\infty}(\mathbb{R}^d)$ with $\inf_{x\in\mathbb{R}^d}g(x)>0$ and $u_{0}\in H^2(\mathbb{R}^{d})$.
Let $(u_{\varepsilon})_{\varepsilon}$ be a very weak solution of (\ref{Equation consistency}). Then, for any regularising families $(g_{\varepsilon})_{\varepsilon}=(g\ast\psi_{\varepsilon})_{\varepsilon}$ and $(u_{0,\varepsilon})_{\varepsilon}=(u_{0}\ast\psi_{\varepsilon})_{\varepsilon}$ for any $\psi\in C_{0}^{\infty}$, $\psi\geq 0$, $\int\psi =1$, such that
\begin{equation}\label{approx.condition}
    \Vert g_{\varepsilon} - g\Vert_{W^{1,\infty}} \rightarrow 0,
\end{equation} the net $(u_{\varepsilon})_{\varepsilon}$ converges to the classical solution of the Cauchy problem (\ref{Equation consistency}) in $L^{2}$ as $\varepsilon \rightarrow 0$.
\end{thm}

\begin{proof}
Let $(u_{\varepsilon})_{\varepsilon}$ be the very weak solution and $u$ the classical one. Let us denote by $V_{\varepsilon}(t,x):=u_{\varepsilon}(t,x)-u(t,x)$. Then $V_{\varepsilon}$ solves the problem
\begin{equation}
    \left\lbrace
    \begin{array}{l}
    i\partial_{t}V_{\varepsilon}(t,x) + \sum_{j=1}^{d}\partial_{x_{j}}\left(g_{\varepsilon}(x)\partial_{x_{j}}V_{\varepsilon}(t,x)\right) = \eta_{\varepsilon}(t,x), \,\,\,(t,x)\in\left[0,T\right]\times \mathbb{R}^{d},\\
    V_{\varepsilon}(0,x)=(u_{0,\varepsilon}-u_{0})(x), \,\,\, x\in\mathbb{R}^{d},
    \end{array}
    \right.
\end{equation}
where
\begin{equation}
    \eta_{\varepsilon}(t,x):= \sum_{j=1}^{d}\partial_{x_{j}}\left[\left(g_{\varepsilon}(x)-g(x)\right)\partial_{x_{j}}u(t,x)\right].
\end{equation}
Repeating the arguments of the previous theorem (uniqueness) and, in particular, using Duhamel's principle, we easily obtain the following estimate for $V_{\varepsilon}$,
\begin{align}
    \Vert V_{\varepsilon}(\cdot,t)\Vert_{L^2} \lesssim \Vert u_{0,\varepsilon}-u_{0}\Vert_{L^2} & +  \sum_{j=1}^{d} \Vert\partial_{x_{j}} g_{\varepsilon}-\partial_{x_{j}}g\Vert_{L^{\infty}} \int_{0}^{T}\Vert\partial_{x_j}u(s,\cdot)\Vert_{L^2} ds\nonumber\\
    & + \Vert g_{\varepsilon}-g\Vert_{L^{\infty}}\int_{0}^{T} \Vert\sum_{j=1}^{d} \partial_{x_j}^{2}u(s,\cdot)\Vert_{L^2} ds. \label{Estimate V_eps}
\end{align}
$u$ is a classical solution, thus, the quantities $\sum_{j=1}^{d} \partial_{x_j}^{2}u(t,x)$ and for all $j=1,...,d$, $\partial_{x_j}u(t,x)$ are bounded in $L^2$. From the other hand, we have that $\Vert g_{\varepsilon} - g\Vert_{W^{1,\infty}} \rightarrow 0$ and $\Vert u_{0,\varepsilon}-u_{0}\Vert_{L^2} \rightarrow 0$ as $\varepsilon\rightarrow 0$. It follows from (\ref{Estimate V_eps}) that $(u_{\varepsilon})_{\varepsilon}$ converges in $L^2$ to the classical solution $u$. This ends the proof of the theorem.
\end{proof}

\end{document}